\documentclass{amsart}

\usepackage{amsmath}
\usepackage{amsfonts}
\usepackage{amssymb}
\usepackage{color}
\usepackage{graphicx}
\usepackage[bookmarksnumbered,colorlinks, linkcolor=blue, citecolor=red, pagebackref, bookmarks, breaklinks]{hyperref}
\usepackage[a4paper, left=2.5cm, right=2.5cm, top=2.5cm, bottom=2.5cm]{geometry}
\usepackage{float}
\usepackage{listings}
\usepackage{tikz}
\usepackage{pstricks,pst-node}
\usetikzlibrary{arrows.meta, positioning}
\usepackage[all]{xy}
\usepackage{pgfplots}
\pgfplotsset{compat=1.18}
\newtheorem{thm}{Theorem}[section]

\newtheorem{lema}[thm]{Lemma}
\newtheorem{prop}[thm]{Proposition}
\theoremstyle{definition}

\theoremstyle{remark}
\newtheorem{rem}[thm]{Remark}

\numberwithin{equation}{section}
\newcommand{\R}{\mathbb R}

\begin{document}
	\title{Optimal Control Strategies for Epidemic Dynamics: Integrating SIR-SI and Lotka--Volterra Models}
	
	\author{Rocio Balderrama, Ignacio Ceresa Dussel and Constanza Sanchez de la Vega}
	
	\address[Rocio Balderrama]{Instituto de Investigaciones Matemáticas Luis A. Santaló (IMAS--CONICET), Departamento de Matemática, Facultad de Ciencias Exactas y Naturales,
		Universidad de Buenos Aires, Ciudad Universitaria, Buenos Aires, Argentina}
	\address[Ignacio Ceresa Dussel]{Instituto de Investigaciones Matemáticas Luis A. Santaló (IMAS--CONICET), Departamento de Matemática, Facultad de Ciencias Exactas y Naturales,
		Universidad de Buenos Aires, Ciudad Universitaria, Buenos Aires, Argentina}
	
	\address[Constanza Sanchez de la Vega]{Instituto de Cálculo, CONICET,
		Departamento de Matemática, Facultad de Ciencias Exactas y Naturales,
		Universidad de Buenos Aires, Ciudad Universitaria, Buenos Aires, Argentina}
	
	\subjclass[2020]{}

	
	\begin{abstract}
		In this work we present a mathematical model that integrates the epidemiological dynamics of a vector-borne disease (SIR–SI) with Lotka--Volterra predator–prey ecological interactions. The study analyzes how the presence of natural predators acts as a biological control mechanism to regulate the vector population and, consequently, disease transmission in host. 
		
		We introduce the concept of the ecological reproduction number, a threshold that links the amplitude of predator--prey cycles to disease persistence, showing that natural control depends critically on the ratio between the maximum vector density and the minimum predator density. In scenarios where natural control is insufficient, we formulate an optimal control problem based on the release of predators. Using the Pontryagin Maximum Principle, we characterize the optimal strategy that minimizes the cumulative number of infected individuals and intervention costs, while simultaneously maximizing the susceptible host population at the end of the time horizon. Numerical simulations validate the effectiveness of the model, showing that external intervention mitigates the epidemic peak and stabilizes the system against the natural oscillations of biological populations.
		
	\end{abstract}
	\maketitle

\section{Introduction}
Epidemiological models have been a powerful tool for understanding the spread of infectious diseases in populations, such as measles, smallpox, influenza, and COVID-19 \cite{Balderrama2022, Balderrama2024, Cooper2020,Duncan1997,Eilertson2019,Kusmawati2021,Rahimi2021}. The most common among these models is the susceptible, infected and recovery model or SIR model, firstly introduced by Kermack and McKendrick \cite{KermackMcKendrick1927} and extensively studied by several authors through the last century, see for example \cite{anderson1992, keeling2008}. The model captures the dynamics of disease transmission by following the temporal evolution of the susceptible, infected, and recovered populations, whose interactions are described by a system of ordinary differential equations represented as
\medskip

\begin{center}
	\begin{tikzpicture}[node distance=3cm, every node/.style={draw, circle, minimum size=1cm}]
	\node (S) {$S$};
	\node (I) [right of=S] {$I$};
	\node (R) [right of=I] {$R$};
	\draw[<->] (S) -- (I) node[midway,below,inner sep=1pt,draw=none] (betah) {};
	\draw[->] (I) -- (R); 
\end{tikzpicture}
\end{center}

However, many infectious diseases are not transmitted directly between humans. Diseases such as Malaria, Chikungunya, and Dengue are spread by mosquitoes. To account for this, several authors \cite{aronna2018interval,Esteva1998,hethcote1989} have introduced additional complexity into classical models by considering the interaction between mosquitoes (vectors) and humans (hosts). In particular, the vector population and the host population are modeled through a coupled system, where the vectors follow a susceptible--infectious structure, leading to the so-called SIR-SI model, represented in the following diagram
\begin{center}
	\begin{tikzpicture}[node distance=3cm, every node/.style={draw, circle, minimum size=1cm}]
		\node (Sh) {$S_h$};
		\node (Ih) [right of=Sh] {$I_h$};
		\node (Rh) [right of=Ih] {$R_h$};
		\node (Sv) [below of= Sh] {$S_v$};
		\node (Iv) [below of= Ih] {$I_v$};
		\draw[->] (Sh) -- (Ih) node[midway,below,inner sep=1pt,draw=none] (betah) {};
		\draw[->] (Ih) -- (Rh); 
		\draw[->] (Sv) -- (Iv) node[midway,below,inner sep=1pt,draw=none] (betav) {};
		\draw[->,blue,dashed] (Iv) -- (betah.north);
		\draw[->,blue,dashed] (Ih) -- (betav.north);
	\end{tikzpicture}
\end{center}
Here $S_h,I_h,R_h$ are the susceptible-infected-removed hosts and $S_v,I_v$ are the susceptible-infected vectors.
	
The purpose of an epidemiological mathematical model is not only descriptive; an equally important goal is to inform and guide epidemic control. Specifically, such models aim to influence disease dynamics by devising strategies that minimize the number of infections, curb pathogen transmission, and optimize the allocation of available resources. The mathematical framework that enables this analysis is provided by \emph{optimal control theory}, which focuses on determining optimal trajectories of the system’s state variables through the adjustment of control variables under given constraints \cite{Evans,Hartl1995}. A substantial body of work has applied optimal control techniques to SIR–SI models, exploring strategies such as vaccination, insecticide use, or sterile insect techniques to reduce transmission \cite{Aronna2020,Bakare2014,Rodrigues2011}. 

In recent years, biological control has attracted growing attention as a complementary or alternative strategy to chemical interventions. In particular, the introduction of natural predators such as dragonflies \cite{Combes2013,Thu1980,Vatandoost2021} or other mosquito predators \cite{Sarfraz2019}, can serve as a natural biological control method. Motivated by this, the present work proposes and analyzes a novel mathematical model that integrates SIR-SI dynamics with a Lotka--Volterra type predator--prey interaction to explore biologically inspired control strategies. 

Specifically, we introduce and analyze a coupled model that combines an SIR–SI epidemic structure with a Lotka–Volterra system describing the interaction between vectors and their natural predators. Unlike previous models (e.g. \cite{Zhou2014}), the predator–prey subsystem in our setting admits periodic solutions that are not contained in any compact set including the origin. As a consequence, standard stability techniques based on Lyapunov or LaSalle arguments relying on compactness cannot be applied directly. Furthermore, the predator–prey oscillations imply that disease persistence is governed not only by epidemiological parameters but also by the magnitude of the underlying ecological cycles. To formalize this dependence, we introduce an ecological reproduction number $R_0(k_0)$, where $k_0$ indexes the invariant Lotka–Volterra orbits of the vector–predator subsystem. That is, the parameter $k_0$ measures the amplitude of predator–prey oscillations and determines the ratio between the maximal vector density and the minimal predator density along a given ecological cycle, thereby allowing us to assess whether natural predation alone can suppress disease transmission.
When natural predation alone cannot effectively limit disease transmission, we introduce an optimal control framework based on the managed release of predators. The objective is to identify a release strategy that balances the economic cost of intervention with its epidemiological benefits, reducing the burden of infection over time while preserving a large susceptible host population at the final time.

The control variable $u(t)$ represents the rate at which captive-bred dragonflies are released into the system. The performance of a given strategy is quantified by an objective functional that incorporates both costs and benefits. The cost of predator release is modeled either by a quadratic term $c\,u^2(t)$ or by a linear term $q\,u(t)$ over the time interval $[0,T]$. Epidemiological outcomes are accounted for by penalizing the cumulative number of infected hosts and by rewarding the number of susceptible hosts at the final time.

The paper is organized as follows: in Section \ref{section 2} we presents the proposed model; Section \ref{section 3} analyzes its dynamics; Section \ref{section 4} establishes the existence and characterization of an optimal control; and Section \ref{section 5} illustrates the model behavior through numerical simulations, providing practical insights for real-world applications. Finally, in Section \ref{section 6} we discuss some conclusions.

\section{The model}\label{section 2}
We begin this section by introducing the mathematical model under consideration. The system describes the interaction between a human host population, a vector population (e.g., mosquitoes), and a natural predator population that acts as a biological control mechanism. The host population follows a classical SIR-type structure, while the vector dynamics are governed by susceptible–infected classes subject to predation, and the predator population, named with $D$, evolves according to its own growth–mortality balance. This interaction is represented by the following diagram:
\begin{center}
	\begin{tikzpicture}[node distance=3cm, every node/.style={draw, circle, minimum size=1cm}]
		\node (Sh) {$S_h$};
		\node (Ih) [right of=Sh] {$I_h$};
		\node (Rh) [right of=Ih] {$R_h$};
		\node (Sv) [below of= Sh] {$S_v$};
		\node (Iv) [below of= Ih] {$I_v$};
		\node (D)  [right of=Iv] {$D$};
		\draw[->] (Sh) -- (Ih) node[midway,below,inner sep=1pt,draw=none] (betah) {};
		\draw[->] (Ih) -- (Rh); 
		\draw[->] (Sv) -- (Iv) node[midway,below,inner sep=1pt,draw=none] (betav) {};
		\draw[->,blue,dashed] (Iv) -- (betah.north);
		\draw[->,blue,dashed] (Ih) -- (betav.north);
		\draw[<->,red,dashed, bend right=-45] (Iv) to node[midway,below,inner sep=1pt,draw=none] (eta) {} (D) ;
		\draw[<->,red,dashed, bend right=-45] (Sv) to (D);
	\end{tikzpicture}
	
\end{center}

 The complete set of equations is
\begin{equation}\label{modelocompleto}
	\begin{cases}
		\dot{S_h} & = \mu_h N_h - b \dfrac{\beta_{h}}{N_h} S_h I_v - \mu_h S_h,\\
		\dot{I_h} & = b \dfrac{\beta_{h}}{N_h} S_h I_v - (\gamma+\mu_h) I_h, \\
		\dot{R_h} & = \gamma I_h - \mu_h R_h,\\
		\dot{S_v} & = \mu_v (S_v+I_v) - \left( b \dfrac{\beta_{v}}{N_h} I_h + \alpha D \right) S_v, \\
		\dot{I_v} & = b \dfrac{\beta_{v}}{N_h} S_v I_h - \alpha D I_v, \\
		\dot{D} & = \eta (S_v + I_v) D - \mu_D D.
	\end{cases}
\end{equation}
We consider a human population divided into susceptible $S_h(t)$, infected $I_h(t)$, and recovered $R_h(t)$ individuals, with total population $N_h = S_h + I_h + R_h$ assumed constant. This reflects balanced births and deaths, implemented through the parameter $\mu_h$, which represents both natural birth and death rates. New individuals enter the susceptible class at rate $\mu_h N_h$, and all compartments experience natural mortality at rate $\mu_h$. We assume that the disease does not cause additional mortality and that recovered individuals acquire permanent immunity.

Disease transmission occurs through a vector, such as a mosquito. Each vector bites on average $b$ individuals per unit time. The probability that a bite from an infected vector transmits the infection to a susceptible human is $\beta_h$, while a susceptible vector becomes infected by biting an infected human with probability $\beta_v$. This gives the frequency-dependent forces of infection $b \frac{\beta_h}{N_h} I_v(t)$ on humans and $b \frac{\beta_v}{N_h} I_h(t)$ on vectors. Infected humans recover at rate $\gamma$, while all individuals are subject to natural mortality at rate $\mu_h$, so the outflow from $I_h$ is $(\gamma + \mu_h) I_h$.

The vector population consists of susceptible $S_v(t)$ and infected $I_v(t)$ individuals. Births occur proportionally to the total vector population $N_v = S_v + I_v$ at rate $\mu_v N_v$, and all newborn vectors are assumed susceptible, since vertical transmission is neglected. Vectors are removed by predation from a natural enemy population $D(t)$ at rate $\alpha$ per predator, affecting both susceptible and infected vectors. Predators grow in proportion to the availability of prey with conversion efficiency $\eta$ and experience natural mortality at rate $\mu_D$. We assume that vectors do not experience disease-induced mortality.

When aggregated as $N_v = S_v + I_v$, the joint dynamics of vectors and predators form a classical Lotka--Volterra system:
\begin{equation}\label{eq: LV}
\begin{cases}
	\dot{N_v} = (\mu_v - \alpha D) N_v,\\[2mm]
	\dot{D} = (\eta N_v - \mu_D) D,
\end{cases}
\end{equation}
where mosquitoes act as prey and predators regulate their population, capturing the ecological basis of biological control.

In summary, the model integrates SIR dynamics of humans with ecological interactions between vectors and predators, accounting explicitly for demographic turnover, frequency-dependent transmission, permanent immunity, absence of disease-induced mortality, and predation-based vector control.

 Moreover, as $N_h$ is constant we can remove one dimension of the problem \eqref{modelocompleto} and simplifies the model to
	\begin{equation}\label{eq: modelo simple}    	
		\begin{cases}
			\dot{S_h} & = \mu_h N_h-b \frac{\beta_{h}}{N_h} S_h I_v-\mu_hS_h,\\
			\dot{I_h} & = b \frac{\beta_{h}}{N_h} S_h I_v - (\gamma+\mu_h) I_h, \\
			\dot{S_v} & = \mu_v (S_v+I_v) -(b \frac{\beta_{v}}{N_h}  I_h  + \alpha  D)S_v, \\
			\dot{I_v} & = b \frac{\beta_{v}}{N_h} S_v I_h - \alpha DI_v, \\
			\dot{D} & =  {\eta} (S_v + I_v ) D -\mu_D D .
		\end{cases}
	\end{equation}

\section{Dynamics of the system} \label{section 3}
In this section, we analyze the dynamical properties of the coupled SIR–SI–Lotka–Volterra system. Our goal is to understand how the interaction between epidemic variables and ecological predator–prey dynamics determines the qualitative behavior of solutions, including disease extinction, persistence, and long-term oscillations. Particular attention is paid to the role of the vector–predator subsystem, whose Lotka–Volterra structure induces invariant periodic orbits that strongly influence the epidemic dynamics.

We begin by establishing basic properties of the system, such as positivity, boundedness, and the characterization of equilibrium points. We then study the stability of disease-free and endemic states, highlighting the limitations of classical stability techniques in the presence of non-compact invariant sets. This leads naturally to the introduction of suitable threshold quantities, including the basic reproduction number and its ecological counterpart, which incorporate the effect of predator–prey oscillations on disease transmission.

\subsection{Positive invariance and equilibrium points} 
\begin{prop} \label{pr: soluciones globales}
		Given non negative initials values, the solutions of system \eqref{eq: modelo simple} are non-negative, bounded for  all $t\geq 0$ and therefore globally defined. 
\end{prop}
	\begin{proof}
		Observe that if $ S_h(t) = 0 $, then $ \dot{S_h}(t) = \mu_h N_h > 0 $ proving that $S_{h}(t)\ge 0$ for all $t\ge 0$, whenever $S_{h}(0)\ge 0$. 
		
		For $D(t)$ we have the following integral expression 
		$$D(t) = D(0)e^{\int_0^t  {\eta}(S_v(s)+I_v(s))-\mu_D(s)\,ds}$$ 
		Therefore, $ D $ is non-negative if it has non-negative initial conditions. Moreover, since $(N_v,D)$ is a solution of the Lotka-Volterra equation  \eqref{eq: LV} we have that $N_v\geq 0$. Hence if $S_v=0$ is immediate that $I_v\geq0$. Therefore, if $I_v=0$ the solutions $S_v$ and $I_v$ remains zero for all times. On the other hand, if $I_v>0$, using that $\dot{S_v}  = \mu_v I_v> 0$ we obtain that $S_v>0$.
		
		To study the the positiveness of $I_h$ and $I_v$ we observe the following sub-system 
		\begin{align*}
			\dot{I}_h& = - (\gamma + \mu_h) I_h + \beta_{h} S_h I_v,\\
			\dot{I}_v & = \beta_{v} S_v I_h -   \alpha D I_v.		
		\end{align*}
	As the inputs $S_h(t)$, $S_v(t)$, and $D(t)$ are non-negative, and the cross partial derivatives with respect to the other variables are positive, the system is a \emph{cooperative monotone system}. 
	By \cite[Proposition~1.1]{Smith1995}, non-negative initial conditions $I_h(0), I_v(0) \ge 0$ imply that $I_h(t), I_v(t) \ge 0$ for all $t \ge 0$. A cooperative monotone system, also called an \emph{order-preserving system}, is one in which increasing any variable does not decrease the rate of change of the others. 
	The interested reader is referred to \cite{Smith1995,SmithWaltman1995} for further details on monotone systems.
			
Finally, as $ N_h(t) $ is constant and since
$$
(S_h + I_h)'(t) = \mu_h N_h - \gamma I_h - \mu_h(S_h + I_h) \leq \mu_h N_h - \mu_h(S_h + I_h),
$$
we deduce that $ S_h + I_h $ is bounded, which implies that both $ S_h $ and $ I_h $ are also bounded. Similarly, from \eqref{eq: LV}, since $ N_v $ and $ D $ satisfy a Lotka-Volterra system, their solutions are periodic orbits and thus bounded for each initial condition. Consequently, we conclude that $ S_v $ and $ I_v $ are bounded as well.

From this analysis, it follows that the solutions of system \eqref{eq: modelo simple} are globally defined, completing the proof.

	\end{proof}
	In the following, we focus on identifying and analyzing the equilibrium points of system \eqref{eq: modelo simple}. Particular attention will be given to the disease-free and endemic equilibria. Additionally, we introduce the basic reproduction number, $R_0$, which serves as a threshold parameter determining whether the infection can invade and persist in the population. To simplify the analysis and reduce the number of parameters, we will employ adimensionalization techniques, rescaling the system to work with dimensionless variables and parameters. The value of $R_0$ will guide the classification of equilibria and the study of their stability properties.
	
	First, we perform a time rescaling using the mosquito birth rate	
	$$\tau=\mu_v t,$$ so that $\frac{d}{dt} = \mu_v \frac{d}{d\tau}$. We also introduce the following dimensionless variables
	\begin{equation}
		s_h=\frac{S_h}{N_h};\	i_h=\frac{I_h}{N_h};\ 	s_v=\frac{S_v}{N_v^*};\ i_v=\frac{I_v}{N_v^*};\ d=\frac{D}{D^*},
	\end{equation}
	where $N_v^*=\frac{\mu_D}{\eta}$ and $D^*=\frac{\mu_v}{\alpha}$. Now, taking account of this and defining
	$$
	B_h=\frac{b\beta_h\mu_D}{\mu_vN_h\eta};\,\, B_v=\frac{b\beta_v}{\mu_v};\,\, \widehat{\mu_h}=\frac{\mu_h}{\mu_v};\,\, \widehat{\gamma}=\frac{\gamma}{\mu_v};\,\, \widehat{\mu_D}=\frac{\mu_D}{\mu_v},
	$$
	we can rewrite the system as follows
	\begin{equation}
		\label{eq: modelo SILV red}
		\begin{cases}
			\frac{\partial s_h(\tau)}{\partial \tau} & = \widehat{\mu_h}(1-s_h) -B_h s_h i_v,\\
			\frac{\partial i_h(\tau)}{\partial \tau}& = B_h s_h i_v - (\widehat{\gamma}+\widehat{\mu_h}) i_h, \\
			\frac{\partial s_v(\tau)}{\partial \tau} & =  (s_v + i_v) -B_v s_v i_h   - d s_v, \\
			\frac{\partial i_v(\tau)}{\partial \tau} & = B_v s_v i_h - d i_v, \\
			\frac{\partial d(\tau)}{\partial \tau} & =  \widehat{\mu_D} d(s_v + i_v -1 )  .
		\end{cases}
	\end{equation}
	
	\begin{rem}[Simplifying the equations]
		The original parameters had units like `` per day " for rates or combinations involving hosts, mosquitoes and predators. After rescaling, everything becomes nice unitless numbers that combine these physical quantities (like $B_h$ combining biting rate, transmission probability, and mosquito lifespan). This cleaner form helps us see what really matters in the disease dynamics without getting distracted by measurement units.
	\end{rem}
	\begin{rem}
	\textbf{Notation:} For simplicity, we will omit the hat notation ($\hat{\cdot}$) for rescaled parameters in the dimensionless equations, as their dimensionless nature is clear from context. Moreover, we will revert to using capital letters for variable names. That is
	\begin{equation}
		\label{eq: modelo SILV}
		\begin{cases}
			\dot{S_h} & = \mu_h(1-S_h)-B_hS_hI_v\\
			\dot{I_h} & = B_h S_h I_v - (\gamma+\mu_h) I_h, \\
			\dot{S_v} & =  (S_v+I_v) - B_v S_v I_h  -  DS_v, \\
			\dot{I_v} & = B_v S_v I_h - DI_v, \\
			\dot{D} & =   (S_v + I_v )\mu_D D -\mu_D D .
		\end{cases}
	\end{equation}
\end{rem}
Moreover the Lotka Volterra subsystem with variables $N_v,D$ reads
	\begin{equation}\label{eq: LV simple}
\begin{cases}
	\dot{N_v} = (1 - D) N_v,\\[2mm]
	\dot{D} = \mu_D (N_v -1) D,
\end{cases}
\end{equation}

	The system \eqref{eq: modelo SILV} has a trivial disease-free equilibrium
	\begin{equation*}\label{eq: free-of-dis-1}
	E_1=( S_h: 1, I_h: 0, S_v:0, I_v: 0, D:0)
	\end{equation*}
	and a non-trivial disease-free equilibrium  
		\begin{equation*}\label{eq: free-of-dis-2}
	E_2=( S_h: 1 , I_h: 0, S_v: 1 , I_v: 0, D:1 )
	\end{equation*}
	which corresponds to the absence of infected vectors and hosts, and an endemic equilibrium point,  
	$$
	E_e=\begin{cases}
		S_h^* = \dfrac{B_v \mu_h + \gamma + \mu_h}{B_v (B_h + \mu_h)}, \\
		I_h^* = \dfrac{\mu_h (B_h B_v - \gamma - \mu_h)}{B_v (B_h + \mu_h)(\gamma + \mu_h)}, \\
		S_v^* = \dfrac{(B_h + \mu_h)(\gamma + \mu_h)}{B_h (B_v \mu_h + \gamma + \mu_h)},\\
		I_v^* = \dfrac{\mu_h (B_h B_v - \gamma - \mu_h)}{B_h (B_v \mu_h + \gamma + \mu_h)}, \\
		\\
		D^* = 1.
	\end{cases}
	$$
	with nonnegative coordinates whenever 
	\begin{align}
		\label{eq: condicion_EE}
		B_hB_v \ge  (\gamma  +\mu_h).
	\end{align}
	We note that that equilibriums $E_2$ and $E_e$ derived into the only nontrivial equilibrium for the Lotka Volterra system given by $(N_v,D)=(1,1)$.

In what follows, we derive the basic reproduction number for this model and analyse the stability of the equilibrium points.

\subsection{Dynamics of $I_h$ and $I_v$}

We begin this section by defining 
the basic reproduction number
\begin{equation}\label{R0}
R_0=\sqrt{\frac{B_h B_v }{ (\gamma + \mu_h)  }}.
\end{equation}

We first prove non stability for $E_1$ and analyse the local stability of the infective compartments for $E_2$.

\begin{prop}\label{local stability}
The trivial disease-free equilibrium $E_1$ is unstable, whereas the non trivial disease-free equilibrium $E_2$ is unstable if $R_0 > 1$ and for $R_0 < 1$, for initial conditions near ${E}_2$,
 the infected populations $I_h(t)$ and $I_v(t)$ goes to zero. Furthermore, if $R_0<1$ the only equilibrium points are the two disease-free equilibria $E_1$ and $E_2$.
\end{prop}
\begin{proof}
As it is standard for studying the stability of the equilibrium points for these models, we rearrange the system in variables $I_h,I_v,S_h,S_v,D$ obtaining the equilibriums 
$\bar{E}_1=(0,0,1,0,0)$ and  $\bar{E}_2=(0,0,1,1,1)$.
We now  compute the  differential matrix for the function $F(I_h,I_v,S_h,S_v,D)$ evaluated on $\bar{E}_1$, obtaining
$$
DF(\bar{E}_1) =
\begin{bmatrix}
	-(\gamma + \mu_h) & B_h & 0 & 0& 0 \\
	0  & 0 & 0 & 0 & 0 \\
	0 & -B_h & -\mu_h & 0 & 0 \\
	0& 1&0 & 1 &  0 \\
	0 &  0 & 0 & 0 & -\mu_D
\end{bmatrix}
$$ 
proving that $\bar{E}_1$ is a non-hyperbolic, non-stable equilibrium since it has a zero and a positive eigenvalue. Now, for $\bar{E}_2$ we obtain
$$
DF(\bar{E}_2) =
\begin{bmatrix}
-(\gamma + \mu_h) & B_h & 0 & 0& 0 \\
B_v  & -1 & 0 & 0 & 0 \\
0 & -B_h & -\mu_h & 0 & 0 \\
-B_v& 1&0 & 0&  - 1 \\
0 &  \mu_D & 0 & \mu_D & 0
\end{bmatrix}.
$$

 If we call 
\[
F=\begin{bmatrix}
	0 & B_h \\
	B_v& 0
\end{bmatrix}, \quad 
V=\begin{bmatrix}
	\gamma+\mu_h &0 \\
 0& 1
\end{bmatrix}, \quad 
J_4=\begin{bmatrix}
-\mu_h & 0 &0\\
0& 0& - 1\\
0 &  \mu_D & 0
\end{bmatrix},
\]
we see that $J_4$ has a negative eigenvalue and two eigenvalues with real part iqual to zero.
Thus we cannot apply the next-generation matrix as 
developed in \cite{Driessche2002}. Despite of that, we see that
\begin{equation*}
R_0=\sqrt{\frac{B_h B_v }{ (\gamma + \mu_h)  }}=\rho(FV^{-1})
\end{equation*}
the spectral radius of the matrix $FV^{-1}$. This gives us the following condition:
the eigenvalues of the submatrix
$$F-V = \begin{bmatrix}
-(\gamma + \mu_h) &B_h\\
B_v & -1 
\end{bmatrix} $$
has negative real part if and only if
$R_0 < 1$ and positive real part if and only if $R_0>1$. Thus, we can conclude that if $R_0>1$, the equilibrium $\bar{E}_2$ is unstable. Moreover, note that we can rewrite the infective coordinates of the endemic equilibrium $E_e$ as follows:
\begin{align*}
I_h^* &= \dfrac{\mu_h (R_0^2-1)}{B_v (B_h + \mu_h)} \\
I_v^* &= \dfrac{\mu_h (R_0^2-1)(\gamma + \mu_h)}{B_h (B_v \mu_h + \gamma + \mu_h)},
\end{align*}
from where we have that $R_0>1$ if and only if condition \eqref{eq: condicion_EE} for the existence of $E_e$ is fullfiled.

Also, since 
\begin{equation}\label{matriz infec}
	\begin{bmatrix} 
\dot{I_h} \\ \dot{I_v} \end{bmatrix} = 
\begin{bmatrix}
-(\gamma + \mu_h) & B_h S_h\\
B_v S_v & -D
\end{bmatrix}. \begin{bmatrix} 
{I_h} \\ {I_v} \end{bmatrix}
\end{equation}
for initial data $(I_h(0),I_v(0),S_h(0),S_v(0),D(0))$ near $\bar{E}_2$ we have that the matrix that defines the linear evolution is near $F-V$ proving that if $R_0<1$, the variables $I_h(t)$ and $I_v(t)$ of these solutions goes to zero as $t\to \infty$. Moreover, being the solutions of $I_h$ uniformly bounded, we have that $R_h(t) \to 0$ and thus $S_h(t)\to 1$ as $t\to \infty$.
\end{proof}

We now analyse the global stability of the non-trivial disease-free equilibrium $E_2$.
We consider the Lyapunov functional of the Lotka--Volterra subsystem
$$
V_{LV}(N_v,D) = \ln(N_v)-N_v+ \ln(D)-D.
$$
Since $V_{LV}$ remains constant along trajectories, the level sets of $V_{LV}$ describe the orbits of the $(N_v,D)$ dynamics. 
Also, since the function $f(x)=\ln(x)-x$ satisfies $f(x)\leq -1$ for all $x>0$, with equality only at $x=1$, it follows that
$$
V_{LV}(N_v,D) = \big(\ln(N_v)-N_v\big)+\big(\ln(D)-D\big) \leq -2.
$$
Given $k_0\leq -2$ we define the domain 
\[
\Omega^+_{k_0}=\left\{ (S_h,I_h,S_v,I_v,D)\in \mathbb{R}^5_+: \ S_h+I_h\leq 1,\ 
V_{LV}(S_v+I_v,D) \leq k_0 \right\}
\]
consisting of the compact region, bounded by the level set $V_{LV}(S_v+I_v,D) = k_0$ for variables $N_v$ and $D$.
Hence,
\begin{itemize}
	\item $k_0=-2$ corresponds to the coexistence equilibrium $(N_v,D)=(1,1)$,
	\item $k_0>-2$ is not possible,
	\item $k_0<-2$ corresponds to closed orbits surrounding $(1,1)$.
\end{itemize}

We observe that $\Omega^+_{k_0}$ is positively invariant, from Proposition \ref{pr: soluciones globales} and the fact that conditions $S_h+I_h\leq 1$ and $V_{LV}(S_v+I_v,D)\leq k_0$ are preserved under the flow of the system.

Moreover, given $k_0\leq-2$ there exist positives constants such that  $0 <a(k_0)\leq D\leq b(k_0)$ and $0 <a(k_0)\leq N_v\leq b(k_0)$.
\begin{center}
\begin{tikzpicture}
	\begin{axis}[
		axis lines=middle,
		xlabel=$x$,
		ylabel={$y$},
		ymin=-3, ymax=1,
		xmin=-0.5, xmax=4,
		xtick={0.155,3.1},
		xticklabels={{\hspace{1.8em}\tiny$a(k_0)$},{\hspace{-1.5em}\tiny$b(k_0)$}}, 
		ytick={-1,0},
		grid=major,
		width=10cm, height=6cm
		]
		
		\addplot[domain=0:4, samples=100, thick, blue] {ln(x)-x};
		\node at (axis cs:2.3,-1) {$y = \ln(x) - x$};
				
		\addplot[dashed, red] coordinates {(0,-2) (4,-2)};
		
		\draw[dashed] (axis cs:0.155,-3) -- (axis cs:0.155,2);
		\draw[dashed] (axis cs:3.1,-3) -- (axis cs:3.1,2);
		\draw (axis cs:1,-3) -- (axis cs:1,2);
		
		\addplot[only marks, mark=*, red] coordinates {(3.1,-2)};
		\addplot[only marks, mark=*, red] coordinates {(0.155,-2)};
		
		\node[anchor=south east] at (axis cs:1,0.2) {\hspace{1em}\tiny $x=1$};
		
	\end{axis}
\end{tikzpicture}
\end{center}
In order to prove global stability for the infective variables, we will need a stronger conditions than $R_0<1$. Given $k_0<-2$, we define
	$$
	R_0(k_0) := \frac{B_h B_v}{\gamma+\mu_h}\,\frac{b(k_0)}{a(k_0)}.
	$$
Note that since $0<a(k_0)<b(k_0)$, then $R_0(k_0)>R_0^2$. 
\begin{prop}\label{Stability}
	Let $k_0 < -2$ and consider the invariant set $\Omega^+_{k_0}$. 
	If $R_0(k_0)<1$, then the non-trivial disease-free equilibrium $E_2$ is stable in $\Omega^+_{k_0}$ and $I_h,I_v \to 0$ as $t\to \infty$.
\end{prop}
\begin{proof}
Observe that
	$$
	\begin{cases}
			\dot{I}_h = B_h S_h I_v - (\gamma + \mu_h) I_h&\leq B_h I_v - (\gamma + \mu_h) I_h \\
			\dot{I}_v = B_v S_v I_h -  D I_v&\leq B_v b(k_0) I_h -a(k_0) I_v.
		\end{cases}
	$$
	We define the Metzler matrix
		\begin{equation}\label{matriz Metzler}
	M=
\begin{bmatrix}
-(\gamma + \mu_h) & B_h \\
B_v b(k_0) & -a(k_0)
\end{bmatrix}
\end{equation}
which has negative trace and 
	$$det(M)= (\gamma + \mu_h)a(k_0) - B_h B_v  b(k_0) >0,
	$$
proving that both eigenvalues has negative real part. If we consider the constant coefficient system $z'(t)=Mz(t)$, with initial data $(I_h(0),I_v(0))$, since $M$ is a Metzler matrix, we have that $I_h(t)\le z_1(t)$ and $I_v(t)\le z_2(t)$, yielding 
both solutions goes to zero as $t$ goes to infinity.
\end{proof}

 In this sense, $R_0(k_0)$ can be interpreted as an \emph{ecological reproduction number}, since the ratio $\frac{b(k_0)}{a(k_0)}>1$ plays a crucial role measuring the maximum quantity of vectors over the mininum quantity of predators. Note that this ratio $\frac{b(k_0)}{a(k_0)}\to +\infty$ as $k_0\to -\infty$ showing that for $k_0$ negative with absolut value large enough, the condition for natural control of infected hosts and vectors will not be achived. Biologically, this is reasonable, as it reflects a situation where vectors outnumber their predators, providing the necessary impetus to implement optimal control strategies.

\smallskip
Finally, to study the dynamics of the endemic equilibrium, which arises when $R_0>1$, we compute 
the differential matrix evaluated on the rearrange equilibrium $\bar{E}_e=(I_h^*, I_v^*, S_h^*, S_v^*,1)$. Using that $S_v^*+I_v^*=1$ we obtain
$$
DF(\bar{E}_e) =
\begin{bmatrix}
	-(\gamma + \mu_h) & B_h S^*_h & B_h I^*_v  & 0 & 0 \\
	B_v S^*_v & -1 & 0 & B_v I^*_h & -I^*_v \\
	0 & -B_h S^*_h & -\mu_h-B_h I^*_v & 0 & 0 \\
	-B_v S^*_v & 1&0 & -B_vI^*_h&  - S^*_v \\
	0 &  \mu_D & 0 & \mu_D  & 0
\end{bmatrix}.
$$

Using the Routh--Hurwitz criterion, one finds that this matrix has three eigenvalues with negative real parts and a pair of complex eigenvalues with zero real part. In this setting, classical approaches such as the Hartman–Grobman theorem or the construction of a Lyapunov functional do not appear to be directly applicable. For the parameter values considered in the numerical simulations, the corresponding equilibrium seems to be stable but not asymptotically stable. Moreover, since the basic reproduction number $R_0$ depends on several parameters, alternative  qualitative behaviors may arise for different parameter regimes.

\section{Optimal control problem}\label{section 4}
		
As mentioned in the introduction, our goal is to maximize the final number of susceptible hosts, meaning hosts that were never infected, by introducing an external source of predators, denoted as $u:[0,T]\to [0,u_\text{max}]$, where $u_{max}$ represents the maximum breeding capacity. Consequently, the last equation in all models will be modified as 
$$\dot{D} = (S_v + I_v - 1)\mu_D D + u(t).$$
		
\begin{rem} \label{re: soluciones acotadas}
	Observe that given initial conditions the solutions of the new system remain non-negative and bounded for all $ t \geq 0 $. To see this, we note that the proof is similar to the case without control (see Proposition \ref{pr: soluciones globales}), since we are adding a non-negative inhomogeneity to the equation for $D$.
\end{rem}  		
	
We want to find a control for the previous system by inputting new predators which 
 penalizes the cost of breeding predators in the interval of time $[0,T]$, while simultaneously minimizes the cumulative number of infected hosts $I_h(t)$ and maximizes the number of susceptible hosts at the final time $aS_h(T)$. 
Since we will consider a maximization problem, 
given a fixed $T$ and weight parameters $c\ge 0$, $q\ge 0$, $r\ge0$ and $a\ge 0$ with $(c,q)\neq (0,0)$, we consider the functional
\begin{equation}
J(u) = \int_0^T -c u^2(t)-qu(t)-rI_h(t) \,dt +a S_h(T).
\end{equation}
The optimal control problem is 
\begin{subequations} \label{eq: OC max}
\begin{align} \label{eq: maximo J}
&\max_{u\in \mathcal{U}}J(u)=\int_0^T -c u^2(t)-qu(t)-rI_h(t) dt +a S_h(T),\\
&\dot{X}(t)=f(X,u), \label{eq: ode_control}
\end{align} 
\end{subequations}
where
\begin{align}
	\label{eq: medibles} 
	{\mathcal U}=\left\{u:[0,T] \to \R \text{ measurable functions such that } u(t)\in [0,u_{max}] \text{ a.e.} \right\},
\end{align} 
 $X=(S_h,I_h,S_v,I_v,D)$ and 
\begin{align}
	f(X,u)= 
	\begin{cases}
&  \mu_h (1-S_h)- B_h S_h I_v,\\
&  B_h S_h I_v - (\gamma+\mu_h) I_h, \\
 &   (S_v+I_v) -B_v I_h S_v  -   D S_v, \\
 &  B_v I_h S_v -  DI_v, \\
 &   (S_v + I_v-1 )\mu_D D + u.
\end{cases}
\end{align}

\subsection{Existence of an optimal control}
We begin by proving the existence of an optimal control in the space of measurable functions $\mathcal U$. 
\begin{prop}
	\label{pr: exist_optimo} The optimal control problem \eqref{eq: OC max} admits a solution on $\mathcal{U}$.
\end{prop}
\begin{proof}
	The proof follows directly from Theorem 23.11 \cite{Clarke2013}.
	We need to check conditions (a)-(f) from this theorem for the functional $-\mathcal{J}$.
	From remark \ref{re: soluciones acotadas} solutions of system \eqref{eq: ode_control} are bounded and therefore item (a) is verified.
	Also, since the control space is compact and convex and the term depending on the state at the final time is continuous, conditions (b), (c), and (f) are fulfilled. 
	From the fact that the integrand for the minimization problem is quadratic with respect to the control when $c>0$ and linear with slope $q> 0$ when $c=0$, it is clearly convex, continuous, and satisfies that 
    $c u^2 + q u+rI_h \ge 0$ for $u\in[0,u_{max}]$, proving item (d). Condition (e) is trivial since initial conditions on the state are given.
	Finally, any $u \in {\mathcal U}$ 
	gives an admissible process for which $-\mathcal{J}$ is finite, completing the hypothesis of Theorem
	23.11.
\end{proof}	

\subsection{Necessary Conditions for an Optimal Control}
\label{sec: necessary conditions}

To derive first-order necessary conditions of optimality for problem \eqref{eq: OC max}, we apply Pontryagin's Maximum Principle. This yields a characterization of optimal controls via the Hamiltonian function and the corresponding adjoint (costate) system.

The Hamiltonian associated to problem \eqref{eq: OC max} is defined as
\begin{equation}\label{HamiltonianM1}
	H(X,u,\vec{p}) = \vec{p} \cdot f(X,u) - \lambda_0 (c u^2 + q u+rI_h),
\end{equation}
where $\vec{p} = (p_1, p_2, p_3, p_4, p_5)$ is the adjoint variable, and $\lambda_0 \geq 0$ is a scalar multiplier. The control $u$ appears quadratically in the cost term, while it affects the dynamics linearly through $f(X,u)$.

Let $(X^*, u^*)$ denote an optimal state-control pair. Then, there exist a scalar $\lambda_0 \geq 0$ and an absolutely continuous function $\vec{p}:[0,T] \to \mathbb{R}^5$, not both identically zero, such that the following conditions are satisfied:

\begin{enumerate}
	\item \textit{Adjoint system.} The components of $\vec{p}$ satisfy the differential equations, for almost every $t \in [0,T]$:
	\begin{align}\nonumber
		\dot{p}_1 & =  (p_1-p_2)B_h I_v + p_1\mu_h\\\nonumber
		\dot{p}_2 & =  p_2 (\gamma+\mu_h)  + (p_3-p_4)B_v S_v+\lambda_0 r.\\\label{pode}
		\dot{p}_3 &=  p_3 (-1 + D) + (p_3- p_4)B_v I_h - p_5  \mu_D D.\\\nonumber
		\dot{p}_4 &= (p_1-p_2) B_h S_h -p_3 + p_4   D - p_5\mu_D  D.\\\nonumber
		\dot{p}_5 &  =p_3  S_v + p_4  I_v - p_5\mu_D  (S_v + I_v) +p_5 \mu_D.
	\end{align}

	\item \textit{Transversality condition.} At the final time $T$, the adjoint variables satisfy:
	\begin{align} \label{pfinal}
		\vec{p}(T) = \lambda_0 (a, 0, 0, 0, 0),
	\end{align}
	which reflects the terminal cost depending only on the first component $S_h(T)$ of the state.
	
	\item \textit{Nontriviality condition.} The pair $(\lambda_0, \vec{p}(t))$ must satisfy
	\begin{align} \label{pmp: nontriviality}
		(\lambda_0, \vec{p}(t)) \neq 0 \quad \text{for all } t \in [0,T],
	\end{align}
	ensuring that the multiplier and the adjoint variable are not simultaneously zero.
	
	\item \textit{Maximality condition.} For almost every $t \in [0,T]$, the optimal control $u^*(t)$ must maximize the Hamiltonian, i.e.,
	\begin{align} \label{pmp: max}
		H(X^*(t), u^*(t), \vec{p}(t)) = \max_{u \in [0, u_{\max}]} H(X^*(t), u, \vec{p}(t)).
	\end{align}

	Substituting the explicit form of $H$, this leads to the inequality
	\begin{equation} \label{pmp: optimality}
		p_5(t) u^*(t) - \lambda_0 (c (u^*(t))^2 + q u^*(t)) \ge p_5(t) u - \lambda_0 (c u^2 + q u) \quad \text{for all } u \in [0, u_{\max}].
	\end{equation}
\end{enumerate}

\begin{lema}\label{le: normal}
The optimal control problem is normal, that is, the multiplier $\lambda_0 \neq 0$.
\end{lema}
\begin{proof}
Suppose for contradiction that $\lambda_0 = 0$. Then, from \eqref{pfinal} it follows that $\vec{p}(T) = 0$, and uniqueness of the solution to the adjoint system implies $\vec{p}(t) \equiv 0$ on $[0,T]$, contradicting the nontriviality condition \eqref{pmp: nontriviality}. Therefore, the problem is normal, and we may set $\lambda_0 = 1$ without loss of generality.
\end{proof}

In this case, the optimality condition \eqref{pmp: optimality} leads to an explicit expression for the optimal control $u^*$ in terms of $p_5$:
\begin{itemize}
	\item If $c > 0$, then $u^*(t)$ is the projection of $\frac{p_5(t) - q}{2c}$ onto the admissible interval $[0, u_{\max}]$, that is,
	
	\begin{align} \label{eq: u*_c_mayor_0}
		u^*(t)&=\left\{ \begin{matrix}0&\text{ for }& p_5(t)<q \\ \frac{p_5(t)-q}{2 c} &\text{ for } & q< p_5(t)<2 c u_{max}+q \\u_{max} & \text{ for} & p_5(t)>2cu_{max}+q\end{matrix} \right\}  \, = \, \max\left\{0,\min\left\{\frac{p_5(t)-q}{2c},u_{max}\right\}\right\} 
	\end{align}
	
	\item If $c = 0$, then the Hamiltonian is linear in $u$, and the optimal control takes a bang-bang form:
\begin{align} \label{eq: u*_c_igual_0}
	u^*(t) & =
	\left\{ \begin{matrix}
		0 & \text{if } p_5(t) < q, \\
		u_{\max} & \text{if } p_5(t) > q.
\end{matrix} \right. 
	\end{align}
\end{itemize}
This expression highlights how the sign and magnitude of the switching function $p_5(t) - q$ determines the control structure.

\section{Numerical simulations}\label{section 5}
In this section, we illustrate the behavior of the proposed epidemiological model and assess the performance of the optimal control derived in Section \ref{sec: necessary conditions}.
Although in previous sections the SIR model was non-dimensionalized to simplify the system, reduce parameters, and facilitate the analysis of equilibria, stability, and the basic reproduction number 
$R_0$, in this section, for numerical simulations and figures we use the original dimensional system, as it preserves the biological interpretation of variables and provides more readable and meaningful graphical results.

The simulations were carried out with the parameter values listed in Table~\ref{tab:params}.  
These correspond to biologically meaningful ranges for mosquito–human interactions and predator–vector dynamics (\cite{Massad2010,Antunes-aronna,Lotka1978}). 
Since the actual parameter values may vary significantly depending on geographical region, climate, and seasonal conditions, we selected representative general values that capture the qualitative behavior of the system.  Accordingly, we consider final times 
$T$ corresponding to a single seasonal period of the outbreak.
This choice allows us to emphasize the mathematical and dynamical features of the model while preserving biological plausibility.

\begin{table}[ht]
	\centering
	\caption{Parameter values used in the numerical simulations.}
	\label{tab:params}
	\begin{tabular}{lll}
		\hline
		\textbf{Parameter} & \textbf{Description} & \textbf{Value}\\ \hline
		$\mu_h$ & Natural mortality rate of humans & $3.4\times 10^{-5}$ $\text{days}^{-1}$ \\
		$\mu_v$ & Birth rate of vectors & $0.0125$  $\text{days}^{-1}$\\
		$\mu_D$ & Natural mortality rate of predators & $0.15$  $\text{days}^{-1}$\\
		$\gamma$ & Recovery rate in humans & $0.14$ $\text{days}^{-1}$\\
		$\alpha$ & Vector mortality induced by predators & $0.3$ $(\text{predators}\times \text{days})^{-1}$ \\
		$\eta$ & Conversion efficiency & $0.1$  $(\text{vectors}\times \text{days})^{-1}$ \\
		$b$ & Bite rate & $0.7$ $\text{hosts}\times \text{days}^{-1}$\\
		$\beta_h$ & Human transmission coefficient & $0.45$ $\text{vectors}^{-1}$ \\
		$\beta_v$ & Vector transmission coefficient & $0.55$ $\text{hosts}^{-1}$ \\
		$u_{\max}$ & Maximum control & $0.5$ \\ \hline
		$a$ & Weight of $S_h(T)$ in the cost functional & $5$ \\
		$c$ & Quadratic weight of $u^2(t)$ in the cost & $0$  or $ 1$ \\
		$q$ & Linear weight of $u(t)$ in the cost & $1$ \\
		$r$ & Weight of $I_h(t)$ in the cost & $5$ \\
		$d$ & Relaxation parameter in control iteration & $0.1$\\ \hline
	\end{tabular}
\end{table}
The full state system and the adjoint system were solved using the \texttt{$solve\_ivp$} function from the \texttt{SciPy} library \cite{Virtanen2020}, which implements an adaptive Runge–Kutta scheme \cite{DormandPrince1980}.
The optimal control was obtained via the classical \emph{forward backward sweep method} \cite{mcasey2012}, where in each iteration, the state system was integrated forward in time, the adjoint system backward in time, and the control was updated according to the optimality condition
\eqref{eq: u*_c_mayor_0} or \eqref{eq: u*_c_igual_0}.
The relaxation factor $d$ was used to stabilize the iterative update
$$
u_{k+1}(t) = d\,u_k(t) + (1-d)\,u^*(t).
$$
For a detailed discussion on the role and choice of the relaxation factor $d$, we refer the reader to \cite{Sharp2021}.
The process was initialized with a $u_0=0.1$ and repeated until the control is stabilized.
The performance index was computed as
\begin{align} \label{eq: functional 2}
J(u) = \int_0^T \!\!\big(c\,u^2(t) + q\,u(t) + r\,I_h(t)\big)\,dt - a\,S_h(T),
\end{align}
using numerical quadrature (\texttt{$scipy.integrate.quad$}).

As the control is directly related to the final time, we will show how the control behaves when different times are considered. In both simulation the initial conditions are $S_h(0)=9$, $I_h(0)=1$, $S_v(0)=9$, $I_v(0)=1$ and $D(0)=0.1$. This values were chosen to ensure that all variables are in the same scale. In Figures 1-3, we consider the representative parameters given in Table \ref{tab:params}.

Figures \ref{fig:T30final} and \ref{fig:T120final} consider the quadratic functional \eqref{eq: functional 2} with $c=1$ showing a control in agreement with \eqref{eq: u*_c_mayor_0}. Also they illustrate how the time horizon of the control problem shapes the qualitative outcome of the intervention.
For short horizons, the optimal strategy primarily acts as a mitigation tool, reducing the impact of the epidemic but without fully stabilizing the underlying ecological dynamics.
In contrast, when the control horizon covers the entire epidemic season, the intervention leads to sustained suppression of the disease and allows seasonal effects to dominate the long-term dynamics of the vector population.
On the other hand, Figure \ref{fig:T120c0final} considers the linear functional $J$ with $c=0$ showing, in this case, a bang-bang control as \eqref{eq: u*_c_igual_0}.

\begin{figure}[ht!]
	\centering
	\includegraphics[width=0.8\linewidth]{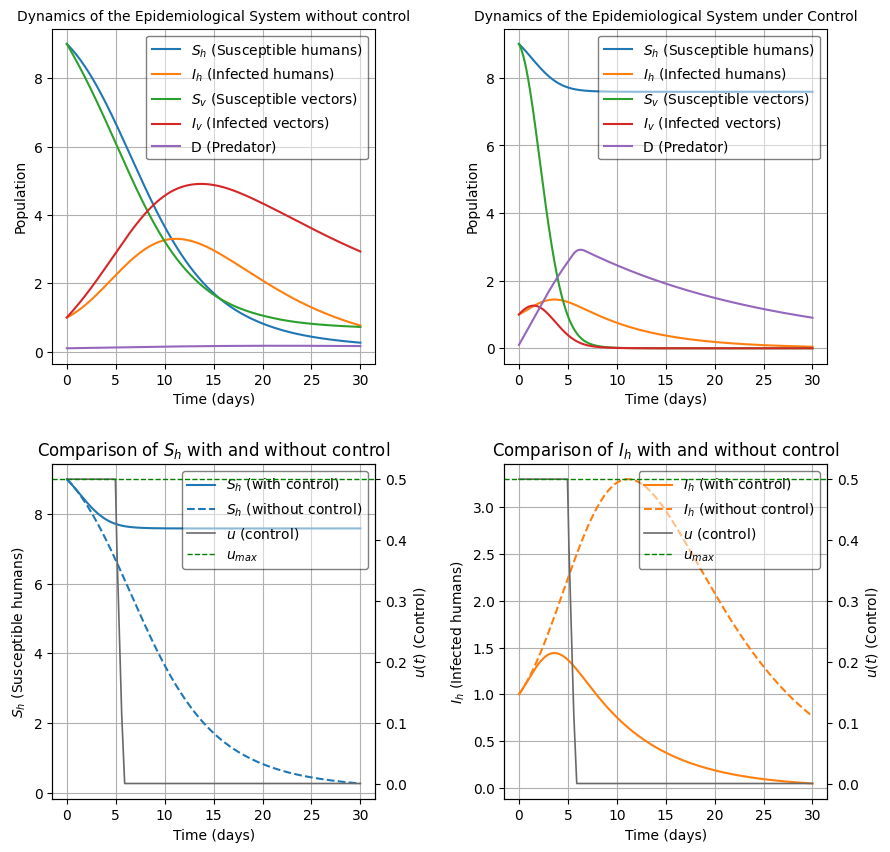}
	\caption{ \small 
Dynamics of the controlled and uncontrolled system for $c=1$ and a short time horizon ($T=30$).
The optimal release of predators substantially mitigates the epidemic by reducing the initial peak and the cumulative number of infected humans.
The control acts mainly during the early phase of the outbreak, weakening transmission through a rapid reduction of the vector population.
However, due to the limited control horizon, the underlying vector--predator oscillations are not fully suppressed, which may allow for a resurgence of the epidemic outside the controlled time window.
}
\label{fig:T30final}
\end{figure}

\begin{figure}[ht!]
	\centering
	\includegraphics[width=0.8\linewidth]{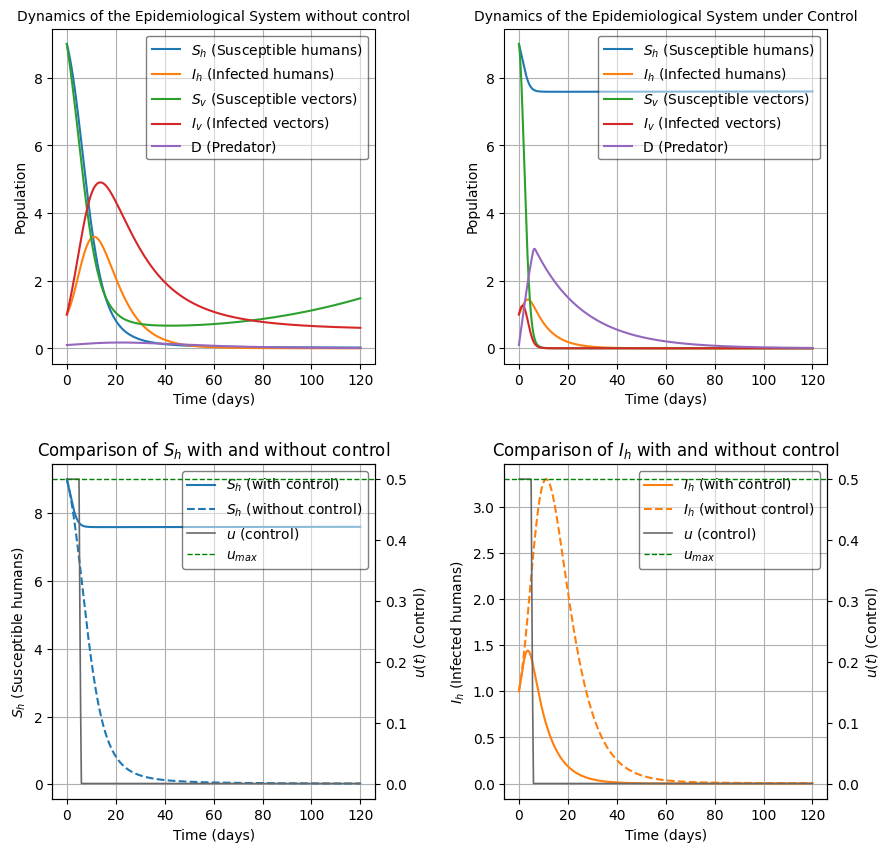}
	\caption{ \small 
Dynamics of the controlled and uncontrolled system for $c=1$ and a long time horizon ($T=120$).
Using the same parameters and cost weights as in Figure~1, the optimal control maintains the number of infected humans at low levels throughout the entire epidemic period.
The extended intervention prevents secondary growth of the vector population during the epidemic phase.
After the epidemic season, seasonal effects naturally drive the mosquito population to extinction, so that no further control is required.}
\label{fig:T120final}
\end{figure}

\begin{figure}[ht!]
	\centering
	\includegraphics[width=0.8\linewidth]{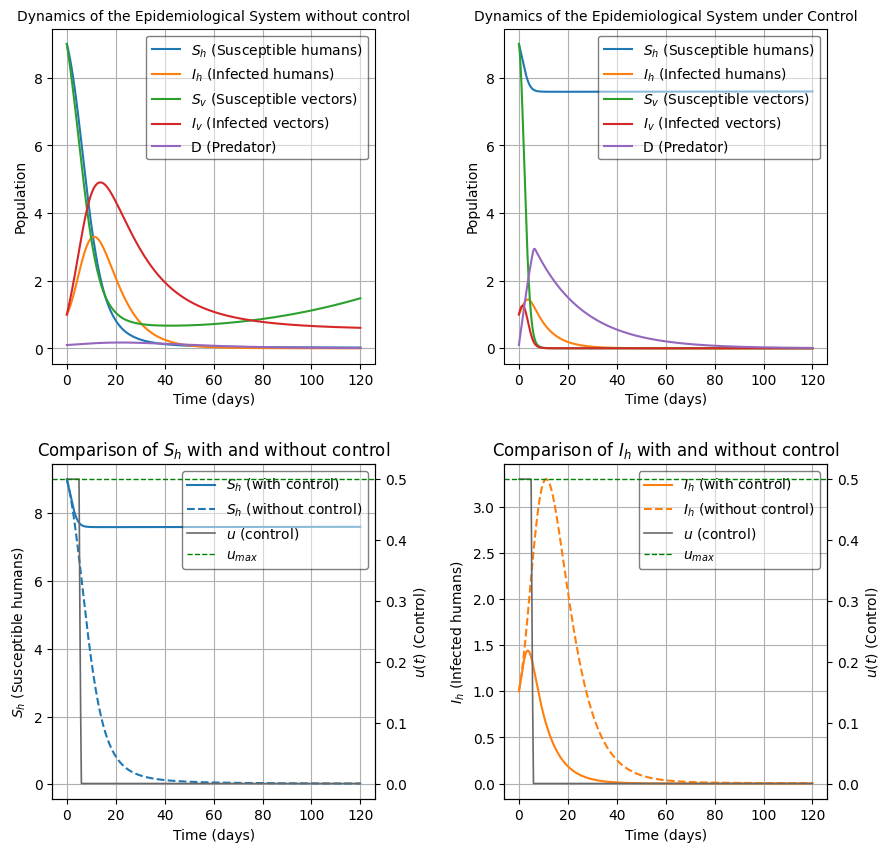}
	\caption{ \small 
Controlled and uncontrolled dynamics for $c=0$, showing a bang--bang optimal control, in agreement with the theoretical predictions of the Pontryagin Maximum Principle (\ref{eq: u*_c_igual_0}).
 The final time is $T=120$ days.}
	\label{fig:T120c0final}
\end{figure}

 \section{Conclusions} \label{section 6}
 
 In this work, we have presented and analyzed a deterministic model that couples the epidemiological dynamics of a human population (SIR) with the ecological dynamics of a vector-predator system governed by Lotka--Volterra interactions. The Lotka--Volterra subsystem in our model exhibits intrinsic periodic solutions, which introduces specific challenges regarding the stability analysis.  
 
 The dynamical analysis revealed that the basic reproduction number, $R_0$, serves as a fundamental threshold for the persistence of the disease. Consistent with the classical results for dengue transmission, the non-trivial disease-free equilibrium is unstable when $R_0 > 1$ and the infective compartments goes to zero locally when $R_0<1$. Moreover, due to the oscillatory nature of the predator-vector interaction, we established a stronger stability condition, the ecological reproduction number $R_0(k_0)$, which depends on the amplitude of the predator-prey cycles. 
 This finding are related to the results of \cite{Zhou2014}, who demonstrated that predation intensity is crucial for disease eradication, by showing that in a Lotka--Volterra framework, the ratio of maximum vector density to minimum predator density within the invariant region determines the feasibility of natural disease control. When this natural disease control is not achieved, that is, the vectors population outnumber predators, an external control is necessary to regulate the outbreak.
 
 Regarding the optimal control problem, we successfully formulated a strategy to maximize the number of susceptible hosts at the final time while minimizing both the cumulative number of infected individuals and the costs associated with breeding predators. Through the application of Pontryagin’s Maximum Principle, we characterized the optimal control laws for quadratic and linear cost structures. The numerical simulations confirm that introducing $u(t)$, captive-bred predators, effectively reduces the infected host population. 
 
 The numerical simulations provide a visual validation of the theoretical optimal control analysis, confirming the practical efficacy of the proposed intervention. By implementing a forward-backward sweep method to solve the coupled state and adjoint systems, the results illustrate how the introduction of an external control $u(t)$—representing the release of captive-bred predators—successfully mitigates the epidemic peak and reduces the cumulative number of infected hosts. The analysis of the control profiles reveals a strong dependency on the chosen time horizon $T$, demonstrating that the optimal strategy dynamically adjusts to the intrinsic oscillations of the predator-vector subsystem to minimize breeding costs while maximizing the final healthy population. 
 
 Together, Figures \ref{fig:T30final} and \ref{fig:T120final} highlight the role of the control horizon in shaping the outcome of the intervention.
While short-term control is effective in mitigating the epidemic peak, it may leave the system vulnerable to rebounds driven by ecological oscillations.
In contrast, long-term control stabilizes the dynamics during the epidemic period and allows seasonal effects to eliminate the vector population naturally, resulting in a more robust and sustained suppression of the disease.

Finally, the numerical experiments with $c=0$ confirm the theoretical characterization of the optimal control.
In the absence of a quadratic penalization on the control effort, the optimal strategy exhibits a bang--bang structure, switching between the extreme admissible values.
This behavior is fully consistent with the predictions of the Pontryagin Maximum Principle and provides an additional numerical validation of the analytical results.

 In summary, this study demonstrates that integrating Lotka-Volterra predation into an SIR-SI model offers a viable biological control mechanism. While natural predation alone may lead to periodic disease prevalence due to the cyclic nature of the vector population, the application of optimal control theory allows for the design of efficient release strategies. These strategies can force the system toward a disease-free state or significantly mitigate the epidemic burden, providing a theoretical foundation for the economic and ecological management of vector-borne diseases.
 
\section*{Acknowledgements}
This work was partially supported by PIP No. 11220220100124CO.

\bibliographystyle{plain}
\bibliography{References.bib}

\end{document}